\documentclass[reqno]{amsart}
\usepackage{amssymb}
\usepackage{amsthm}
\usepackage{bbm}
\usepackage{graphicx}

\usepackage[colorlinks=true]{hyperref}
\usepackage[all]{xypic}

\usepackage[final]{epsfig}
\usepackage{psfrag}
\usepackage{epstopdf}

\newtheorem{thm}{Theorem}[section]
\newtheorem{theorem}[thm]{Theorem}
\newtheorem{proposition}[thm]{Proposition}

\newtheorem{lemma}[thm]{Lemma}
\newtheorem{cor}[thm]{Corollary}
\newtheorem{lem}[thm]{Lemma}
\newtheorem{prop}[thm]{Proposition}

\theoremstyle{definition}
\newtheorem{definition}{Definition}[section]

\theoremstyle{observation}

\theoremstyle{definition}
\newtheorem{example}{Example}[section]

\newcommand{\A}{\mathcal{A}}
\newcommand{\B}{\mathcal{B}}

\newcommand{\M}{\mathcal{M}}

\renewcommand{\S}{\mathcal{S}}
\newcommand{\defin}[1]{{\it #1}}
\newcommand{\R}{\mathbb{R}}
\newcommand{\N}{\mathbb{N}}
\newcommand{\Q}{\mathbb{Q}}

\newtheorem{remark}[thm]{Remark}
{\bf}{\it}
{\bf}{\it}
\newtheorem*{riemann-mapping-theorem}{Riemann Mapping  Theorem}{\bf}{\it}
\newtheorem*{onequarter}{Koebe 1/4  Theorem}{\bf}{\it}
{\bf}{\it}
{\bf}{\it}
{\bf}{\it}
{\bf}{\it}
\newtheorem*{fshb}{Fatou-Shishikura Bound}{\bf}{\it}
{\bf}{\it}
{\bf}{\it}
{\bf}{\it}

\newenvironment{pf*}[1]{\proof[#1]}{\endproof}
\usepackage{euscript}

\newcommand{\cal}[1]{{\mathcal #1}}

\newcommand{\beq}{\begin{equation}}
\newcommand{\eeq}{\end{equation}}

\newcommand{\ve}{\varepsilon}
\newcommand{\de}{\delta}

\newcommand{\be}{\beta}

\newtheorem{defn}{Definition}[section]

\renewcommand{\deg}{\operatorname{deg}}
\newcommand{\riem}{{\hat{\CC}}}

\newcommand{\dist}{\operatorname{dist}}

\newcommand{\cl}{\operatorname{cl}}
\renewcommand{\mod}{\operatorname{mod}}

\newcommand{\eps}{\epsilon}

\numberwithin{equation}{section}
\newcommand{\thmref}[1]{Theorem~\ref{#1}}
\newcommand{\propref}[1]{Proposition~\ref{#1}}

\newcommand{\lemref}[1]{Lemma~\ref{#1}}
\newcommand{\corref}[1]{Corollary~\ref{#1}}
\newcommand{\figref}[1]{Figure~\ref{#1}}

\newcommand{\supp}{\operatorname{Supp}}

\newcommand{\cA}{{\cal A}}

\newcommand{\cM}{{\cal M}}

\newcommand{\cB}{{\cal B}}

\renewcommand{\cH}{{\cal H}}

\renewcommand{\cD}{{\cal D}}

\newcommand{\cS}{{\cal S}}

\newcommand{\PP}{{\Bbb P}}
\newcommand{\CC}{{\mathbb C}}
\newcommand{\RR}{{\mathbb R}}
\newcommand{\EE}{{\mathbb E}}

\newcommand{\ZZ}{{\mathbb Z}}
\newcommand{\NN}{{\mathbb N}}

\newcommand{\QQ}{{\mathbb Q}}

\newcommand{\ignore}[1]{{}}

\title[Computability of Brolin-Lyubich Measure]{Computability of Brolin-Lyubich Measure}
\author{Ilia Binder, Mark Braverman, Cristobal Rojas, Michael Yampolsky}
\thanks{I.B. and M.Y. were partially supported by NSERC Discovery Grants. Part of this work was done while M.B. was a Postdoctoral Fellow at Microsoft Research, New England.}
\date{September 9, 2010}
\begin{document}
\begin{abstract}
Brolin-Lyubich measure $\lambda_R$ of a rational endomorphism $R:\riem\to\riem$ with $\deg R\geq 2$ is the unique invariant measure of maximal entropy $h_{\lambda_R}=h_{\text{top}}(R)=\log d$. Its support is the Julia set $J(R)$. We demonstrate that $\lambda_R$ is always computable by an algorithm which has access to coefficients of $R$, even when $J(R)$ is not computable. In the case when $R$ is a polynomial, Brolin-Lyubich measure coincides with the harmonic measure of the basin of infinity. We find a sufficient condition for computability of the harmonic measure of a domain, which holds for the basin of infinity of a polynomial mapping, and show that computability may fail for a general domain.
\end{abstract}

\maketitle

\section{Foreword}

This paper continues the line of works \cite{BY,BBY1,BBY2,BY-MMJ,BY-book} of several of the authors on algorithmic computability of Julia sets.
In this brief introduction we outline our results and attempt to give a brief motivation for them.

\subsection*{Numerical simulation of a chaotic dynamical system: the modern paradigm}
A dynamical system can be simple, and thus easy to implement numerically. Yet its orbits may exhibit a very complex
behaviour. The famous paper of Lorenz \cite{Lorenz}, for example, described a rather simple nonlinear system of
ordinary differential equations $\bar x'(t)=F(\bar x)$ in three dimension which exhibits {\it chaotic} dynamics.
In particular, while the flow of the system $\Phi_t(\bar x_0)$ is easy to calculate with an arbitrary precision for 
any initial value $x_0$ and any time $t$. However, any error in estimating the initial value $\bar x_0$ grows exponentially
with $t$. This renders impractical attempting to numerically simulate the behaviour of a trajectory of the system for an extended 
period of time: small computational errors are magnified very rapidly. If we recall that the Lorenz system was introduced
as a simple model of weather forecasting, one understands why predicting weather conditions several days in advance
is difficult to do with any accuracy. 

On the other hand, there is a great regularity in the global structure of a typical trajectory of Lorenz system. As was
ultimately shown by Tucker \cite{Tucker}, there exists a set $\cA\subset \RR^3$ such that for almost every initial point $\bar x_0$, the 
limit set of the orbit,
$$\omega(\bar x_0)=\cA.$$
This set is the {\it attractor} of the system \cite{Smale,Milnor-attractor}. Moreover, for any continuous test function $\psi$, the 
time average of $\psi$ along a typical orbit 
$$\frac{1}{T}\int_{t=0}^T \psi(\Phi_{\bar x_0}(t) dt$$
converges to the integral $\int \psi d\mu$ with respect to a measure $\mu$ supported on $\cA$. 

Thus, both the spatial layout and the statistical properties of a large segment of a typical trajectory can be understood, and,
indeed, simulated on a computer: even mathematicians unfamiliar with dynamics have seen the butterfly-shaped picture of  Lorenz attractor $\cA$.
This example summarizes the modern paradigm of numerical study of chaos: while the simulation of an individual orbit for an extended
period of time does not make a practical sense, one should study the limit set of a typical orbit (both as a spatial object and as
a statistical distribution).
A modern summary of this paradigm is found, for example, in the article of J.~Palis \cite{Palis}.

\subsection*{Julia sets as counterexamples, and the topic of this paper}

Julia sets are {\it repellers} of discrete dynamical systems generated by rational maps $R$ of the Riemann sphere $\hat \CC$ of degree $d\geq 2$.
For all but finitely many points $z\in \hat\CC$ the limit of the $n$-th preimages $R^{-n}(z)$ 
coincides with the Julia set $J(R)$. The dynamics of 
$R$ on the set $J$ is chaotic, again rendering numerical simulation of individual orbits impractical. Yet Julia sets are among the
most drawn mathematical objects, and countless programs have been written for visualizing them. 

In spite of this, two of the authors showed in \cite{BY-MMJ} that there exist quadratic polynomials $f_c(z)=z^2+c$ with the following
paradoxical properties:
\begin{itemize}
\item an iterate $f_c(z)$ can be effectively computed with an arbitrary precision;
\item there does not exist an algorithm to visualize $J(f_c)$ with an arbitrary finite precision.
\end{itemize}

This phenomenon of {\it non-computability} is rather subtle and rare. For a detailed exposition, the reader is referred to the
monograph \cite{BY-book}. In practical terms it should be seen as a tale of caution in applying the above paradigm.

We cannot accurately simulate the set of limit points of the preimages $(f_c)^{-n}(z)$, but what about their statistical distribution?
The question makes sense, as for all $z\neq \infty$ and every continuous test function $\psi$, the averages
$$\frac{1}{2^n}\sum_{w\in (f_c)^{-n}(z)} \psi(w)\underset{n\to\infty}{\longrightarrow}\int \psi d\lambda,$$
where $\lambda$ is the {\it Brolin-Lyubich probability measure} \cite{Brolin,Lyubich} supported on the Julia set $J(f_c)$.
We can thus ask whether the value of the integral on the right-hand side can be algorithmically computed with an arbitrary precision.

Even if 
$J(f_c)=\supp(\lambda)$ is not a computable set,  the answer does not {\it a priori} have to be negative. Informally speaking,
a positive answer would imply a dramatic
difference between the rates of convergence in the following two limits:
$$\lim (f_c)^{-n}(z)\underset{\text{\small Hausdorff}}{\longrightarrow} J(f_c)\text{ and }\lim \frac{1}{2^{n}}\sum_{w\in (f_c)^{-n}(z)} \delta_w
\underset{\text{\small weak}}{\longrightarrow} \lambda.$$

The main results of the present paper are the following:

\medskip
\noindent
{\bf Theorem A.} {\it The Brolin-Lyubich measure is always computable.}

\medskip
\noindent
The result of Theorem A is {\em uniform}, in the sense that there is a single algorithm that takes the rational map $R$ as a parameter
and computes the corresponding Brolin-Lyubich measure. Surprisingly, the proof of Theorem A does not involve much analytic machinery. The result follows from the general computable properties of the relevant space of measures. 

Using the analytic tools given by the work of Drasin and Okuyama \cite{okuyama}, or 
Dinh and Sibony \cite{DinhSib}, we get the following:

\medskip
\noindent
{\bf Theorem B.} {\it For each rational map $R$, there is an algorithm $\cA(R)$
that computes the Brolin-Lyubich measure in exponential time.}

\medskip
\noindent
The running time of $\cA(R)$ will be of the form $\exp (c(R)\cdot n)$, where $n$ is 
the precision parameter, and $c(R)$ is a constant that depends only on the map $R$ (but
not on $n$). Theorems A and B are not comparable, since Theorem B bounds the growth
of the computation's running time in terms of the precision parameter, while  Theorem A
gives {\em a single} algorithm that works for all rational functions $R$.

Lastly,
the Brolin-Lyubich measure for a polynomial coincides with the harmonic measure of the complement of the filled Julia set.
As shown in \cite{BY-MMJ} by two of the authors,
the filled Julia set of a polynomial is always computable.
In view of Theorem A, it is natural to ask what property of a computable 
compact set in the plane ensures computability of the harmonic measure of the complement. 
We show:

\medskip
\noindent
{\bf Theorem C.} {\it If a closed set $K\subset \riem$ is computable and uniformly perfect, and has a connected complement, then the harmonic measure of the complement is 
computable}.

\medskip
\noindent
It is well-known \cite{MR92} that filled Julia sets are uniformly perfect. Theorem C thus implies Theorem A in the polynomial case.
Computability of the set $K$ is not enough to ensure computability of the harmonic measure:
we present a counter-example of a computable closed
set with a non-computable harmonic measure of the complement.

 \section{Julia sets of rational mappings}
\label{sec:intro-dyn}

\subsection{Dynamics on the Riemann sphere}
We attempt to summarize here for convenience of a reader, unfamiliar with Complex Dynamics, the
basic facts about Julia sets of rational mappings. An excellent book of Milnor \cite{Mil} presents a detailed
and self-contained introduction to the subject; proofs of most of the facts we state can be found there.

We first recall that the Riemann sphere $\riem$ is the Riemann surface with the topological type of the 2-sphere,
$S^2$. Such a complex-analytic manifold can be constructed by gluing together two copies of the complex plane
$C_1=\CC$, $C_2=\CC$ by identifying $z\in C_1\setminus \{0\}$ with $w=1/z\in C_2$. 
This procedure can be loosely described as adjoining a point at infinity to the complex plane $C_1$ -- we will 
denote $\infty$ the origin in $C_2$ (so that ``$\infty=1/0$''). 
It is convenient sometimes to visualize $\riem$ as the unit sphere 
$$S^2=\{x^2+y^2+z^2=1\}\subset \RR^3.$$
To this end, consider the stereographic projection from the ``north pole'' $(0,0,1)\subset S^2$, which sends $S^2\setminus \{(0,0,1)\}$
to the plane $z=0$ which we naturally identify with $C_1=\CC$ by $z=x+iy$. In this model, the north pole becomes the point at infinity.

The Euclidean metric on $\RR^3$ restricted to $S^2$ is transferred by the stereographic projection to the {\it spherical metric}
on $\CC$. This metric is given by
$$ds^2=\left( \frac{2}{1+|z|^2}\right)|dz|^2.$$
We will refer to the spherical distance as $d_\riem$, as opposed to the usual Euclidean distance $d$.

A rational function $R(z)=P(z)/Q(z)$ induces an analytic covering $\riem\to\riem$ branched at the finitely many {\it critical points}
$\zeta\in\riem$ with $R'(\zeta)=0$. The degree $d$ of this covering is finite, and coincides with the algebraic degree of $R$:
$$d=\max(\deg(P),\deg(Q)),$$
assuming $P$ and $Q$ have no common factors. Every 
 analytic branched covering of $\riem$ of a finite degree is given by a rational
function.

We will consider a rational mapping $R$ of degree $\deg R=d\geq 2$ (that is, non-linear)  as a dynamical system on the Riemann
sphere; and denote $R^n$ the $n$-th iterate of $R$. The $R$-orbit of a point $\zeta$ is the sequence $\{R^n(\zeta)\}_{n=0}^\infty.$
The Julia set is defined as the complement of the set where the dynamics is Lyapunov-stable:

\begin{defn}
Denote $F(R)$ the set of points $z\in\riem$ having an open neighborhood $U(z)$ on which the
family of iterates $R^n|_{U(z)}$ is equicontinuous; that is for every $\eps>0$ there exists $\delta>0$ such that 
if $d_\riem(z,w)<\delta$ then for every $n\in\NN$ one has $d_\riem(R^n(z),R^n(w))<\eps.$
 The set $F(R)$ is called the Fatou set of $R$
and its complement $J(R)=\riem\setminus F(R)$ is the Julia set.
\end{defn}

\noindent
In the case when the rational mapping is a polynomial $$P(z)=a_0+a_1z+\cdots+a_dz^d:\CC\to\CC$$ an equivalent
way of defining the Julia set is as follows. Obviously, there exists a neighborhood of $\infty$ on $\riem$
on which the iterates of $P$ uniformly converge to $\infty$. Denoting $A(\infty)$ the maximal such domain of attraction
of $\infty$ we have $A(\infty)\subset F(R)$. We then have 
$$J(P)=\partial A(\infty).$$
The bounded set $\riem \setminus  A(\infty)$ is called {\it the filled Julia set}, and denoted $K(P)$;
it consists of points whose orbits under $P$ remain bounded:
$$K(P)=\{z\in\riem|\;\sup_n|P^n(z)|<\infty\}.$$

\noindent
For future reference, let us summarize in a proposition below the main properties of Julia sets:

\begin{prop}
\label{properties-Julia}
Let $R:\riem\to\riem$ be a rational function. Then the following properties hold:
\begin{itemize}
\item[(a)] $J(R)$ is a non-empty compact subset of $\riem$ which is completely
invariant: $R^{-1}(J(R))=J(R)$;
\item[(b)] $J(R)=J(R^n)$ for all $n\in\NN$;
\item[(c)] $J(R)$ has no isolated points;
\item[(d)] if $J(R)$ has non-empty interior, then it is the whole of $\riem$;
\item[(e)] let $U\subset\riem$ be any open set with $U\cap J(R)\neq \emptyset$. Then there exists $n\in\NN$ such that
$R^n(U)\supset J(R)$;
\item[(f)] periodic orbits of $R$ are dense in $J(R)$.
\end{itemize}
\end{prop}

\noindent
Let us further comment on the last property. For a periodic point $z_0=R^p(z_0)$
of period $p$ its {\it multiplier} is the quantity $\lambda=\lambda(z_0)=DR^p(z_0)$.
We may speak of the multiplier of a periodic cycle, as it is the same for all points
in the cycle by the Chain Rule. In the case when $|\lambda|\neq 1$, the dynamics
in a sufficiently small neighborhood of the cycle is governed by the Mean
Value Theorem: when $|\lambda|<1$, the cycle is {\it attracting} ({\it super-attracting}
if $\lambda=0$), if $|\lambda|>1$ it is {\it repelling}. All repelling periodic points are in the Julia set,
and all attracting ones are in the Fatou set.

The situation is much more complicated when $|\lambda|=1$; understanding of the local dynamics
in this case is not yet complete.

One of the founders of the subject, P. Fatou, has shown that   
that for a rational mapping $R$ with $\deg R=d\geq 2$
at most finitely many periodic orbits are non-repelling. A sharp bound on their number  depending on $d$ has
been established by Shishikura; it is equal to the number of critical points of $R$ counted with
multiplicity: 

\begin{fshb}
For a rational mapping of degree $d$ the number of the non-repelling periodic
cycles taken together with the number of cycles of Herman rings is at most $2d-2$.
For a polynomial of degree $d$ the number of non-repelling periodic cycles in $\CC$
is at most $d-1$.
\end{fshb}

\noindent
 Therefore, we may refine the last statement of \propref{properties-Julia}:

\begin{itemize}
\item[(f')] $J(R)=\overline{\{\text{repelling periodic orbits of }R\}}.$
\end{itemize}

We also note a useful corollary of \propref{properties-Julia} (e):

\begin{cor}
\label{preimages dense}
Let $w\in J(R)$. Then
$$J(R)=\overline{ \bigcup_{k\geq 0} f^{-k}(w)}.$$
\end{cor}

To conclude the discussion of the basic properties of Julia sets, let us consider the simplest
examples of non-linear rational endomorphisms of the Riemann sphere, the quadratic polynomials.
Every affine conjugacy class of quadratic polynomials has a unique representative of the
form $f_c(z)=z^2+c$, the family
$$f_c(z)=z^2+c,\;c\in\CC$$
is often referred to as {\it the quadratic family}.
For a quadratic map the structure of the Julia set is governed by the behavior of the orbit of the only
finite critical point $0$. In particular, the following dichotomy holds:

\begin{prop}
\label{quadratic-Julia}
Let $K=K(f_c)$ denote the filled Julia set of $f_c$, and $J=J(f_c)=\partial K$. Then:
\begin{itemize}
\item $0\in K$ implies that $K$ is a connected, compact subset of the plane with connected complement;
\item $0\notin K$ implies that $K=J$ is a planar Cantor set.
\end{itemize}
\end{prop}

\noindent
The {\it Mandelbrot set} $\cM\subset \CC$ is defined as the set of parameter values $c$ for which 
$J(f_c)$ is connected.

\ignore{
\noindent
A rational mapping $R:\hat\CC\to\hat\CC$ is called {\it hyperbolic} if the orbit of every critical point of $R$
is either periodic, or  converges to an
(super-)attracting cycle. The term ``hyperbolic'' has an established meaning in dynamics. Its use in this
context is justified by the following proposition:

\begin{prop}
A rational mapping $R$ of degree $d\geq 2$  is hyperbolic if and only if 
there exists a smooth metric $\mu$ defined on an open neighborhood of $J(R)$ and constants $C>0$, $\lambda>1$
such that
$$||DR^n(z)||_\mu>C\lambda^n\text{ for every }z\in J(R), n\in\NN.$$  
\end{prop}

\noindent
As easily follows from Implicit Function Theorem and considerations of local dynamics of an attracting orbit,
hyperbolicity is an open property in the parameter space of rational mappings of degree $d\geq 2$.

\noindent
Considered as a rational mapping of the Riemann sphere, a quadratic polynomial $f_c(z)$ has two critical points:
the origin, and the super-attracting fixed point at $\infty$. In the case when $c\notin \cM$, the orbit of the
former converges to the latter, and thus $f_c$ is hyperbolic. 
The following conjecture is central to the field of dynamics in one complex variable:

\medskip
\noindent
{\bf Conjecture (Density of Hyperbolicity in the Quadratic Family).} {\it Hyperbolic parameters are dense  
in $\CC$.}

}
\subsection{Brolin-Lyubich measure on the Julia set}

\begin{defn}
\label{defn-balanced}
Consider a rational map $R:\riem\to \riem$ of degree $d\geq 2$. We say that a probability measure $\mu$ on $\riem$ is {\it balanced}
(with respect to $R$) if for every set $X\subset \riem$ on which $R$ is injective we have
$$\mu(R(X))=d\cdot\mu(X),$$
that is, the Jacobian of $\mu$ is equal to $d$.
\end{defn}
We see that a balanced measure $\mu$ is necessarily invariant: as most points in $\riem$ have $d$ preimages under $R$,
$$\mu (R^{-1}(X))=\mu(X).$$
However, a rational map has many invariant probability measures (as a simplistic example, for a periodic orbit $z_0\mapsto z_1\mapsto\cdots z_{p-1}\mapsto z_0$
define $\mu=\frac{1}{p}\sum \delta_{z_i}$). On the other hand there is exactly one balanced measure for $R$: the Brolin-Lyubich measure
$\lambda$. Constructed by Brolin \cite{Brolin} for polynomials, and later by Lyubich \cite{Lyubich} for a general rational function, it is supported on the
Julia set $J(R)$. Lyubich showed that for all but finitely many points $z\in\riem$ the weak limit
\begin{equation}
\label{convergence-meas}
\lim_{n\to\infty} \frac{1}{d^n}\sum_{w\in R^{-n}(z)}\delta_w = \lambda.
\end{equation}

In general, given a transformation $T$ of a compact space $X$, denote by $h_\text{top}(T)$ and $h_{\mu}(T)$ the topological and measure-theoretic 
entropies, respectively. The well known \emph{Variational Principle}, tell us that: $$h_\text{top}(T)=\sup_{\mu \in \M_{T}}h_{\mu}(T),$$ where $\M_{T}$ denotes the set of $T$-invariant measures. A measure $\mu$ is called  {\it a measure of maximal entropy} if $h_{\mu}(T)=h_\text{top}(T)$. 
 
Lyubich showed that $\lambda$ is the unique {\it measure of maximal entropy} of $R$:
\begin{thm}[\cite{Lyubich}]
The measure $\lambda$ is the unique measure on $\riem$ for which the metric entropy $h_\lambda(R)$ coincides with the topological entropy of $R$:
$$h_\lambda(R)=h_{\text{top}}(R)=\log d.$$
\end{thm}

Note that for any invariant measure $\mu$ we have
$$\int \text{Jac}_\mu fd\mu\leq d,$$
therefore a measure of maximal entropy is necessarily balanced.

\subsection{Harmonic measure in polynomial dynamics}
\label{sec:23}
A detailed discussion of harmonic measure can be found in \cite{Garnett-Marshall}. Here we briefly recall some of the relevant facts.

Let $G$ be a simply-connected domain  in $\riem$ whose complement $K$ contains at least two points, and $g\in G$. The {\it harmonic measure}
$\omega_{G,g}$ is defined on the boundary $\partial G$. For a set $E\subset \partial G$ it is equal to the probability that
a Brownian path originating at $g$ will first hit $\partial G$ within the set $E$.

To define the harmonic measure for a non simply-connected domain $G\equiv \riem\setminus K$ we have to require that a Brownian path originating in $G$ 
will hit $\partial G$ almost surely, a condition which is satisfied automatically for a simply-connected domain.
A quantitative measure of a likelyhood that such a set will be hit by a Brownian path is defined as follows.
Consider $K\Subset \CC$, and let 
$B_t$ be a Brownian path which is started uniformly at a circle $\{|z|=R\}$ which surrounds $K$.
Denote $\tau$ the first moment when $B_\tau\in K$. The {\it logarithmic capacity} of $K$ is 
$$\text{Cap}(K)=\exp(\EE(\log |B_\tau|)).$$

By way of an example, consider a connected and locally-connected compact set $K\subset \CC$. In this case, $\partial G$ is a continuous image of the unit circle.
In fact, consider the unique conformal Riemann mapping 
$$\psi:G\equiv\riem\setminus K\to \riem\setminus\overline{D_R(0)}\equiv (D_R(0))^c,\text{ with }\psi(\infty)=\infty\text{ and }\psi'(\infty)=1.$$
The quantity $r(G,\infty)\equiv 1/R$ is the {\it conformal radius} of $G$ about $\infty$. 

By a classical theorem of Carath{\'e}odory, $\psi^{-1}$ extends continuously to map $\bar G\to \riem\setminus D_R(0)$. By symmetry considerations,
the harmonic measure $\omega_{(D_R(0))^c,\infty}$ coincides with the Lebesgue measure $\mu$ on the  circle $\partial D_R(0)=\{|z|=R\}$. Conformal invariance of Brownian motion
implies that $\omega_{G,\infty}$ is obtained by pushing forward $\mu$ by $\psi^{-1}|_{\partial D_R(0)}$, and that
$$\text{Cap}(G)=1/r(G,\infty).$$

Consider a polynomial $P:\CC\to\CC$ with $\deg P\geq 2$. The capacity of the filled Julia set $K(P)$ is equal to one. This follows from a classical
result of B{\"o}ttcher when $K(P)$ is connected (see \cite{Brolin} for the general case). Brolin \cite{Brolin} was the first to show that
the balanced measure $\lambda$ of $P$ coincides with the harmonic measure $\omega_{\riem\setminus K(P),\infty}$. 
 \section{Computability}

\subsection{Algorithms and computable functions on integers}
The notion of an algorithm was formalized in the 30's,
independently by Post, Markov, Church, and, most famously, Turing. Each of them  
proposed a model of computation which determines a set of integer functions that can be \emph{computed} by some mechanical or algorithmic procedure. Later on, all these models were shown to be equivalent, so that they define the same class of integer
functions, which are now called \emph{computable (or recursive) functions}.  
It is standard in Computer Science to formalize
an algorithm as a {\it Turing Machine} \cite{Tur}. We will not define it here, and instead will refer an interested
reader to any standard introductory textbook in the subject. It is more intuitively familiar, and provably equivalent,
to think of an algorithm as a program written in any standard programming language.

In any programming language there is only a countable number of possible algorithms. Fixing the language, we can
enumerate them all (for instance, lexicographically). Given such an ordered list $(\cA_n)_{n=1}^\infty$ of all
algorithms, the index $n$ is usually called the \emph{G\"odel number} of the algorithm $\cA_{n}$.

We will call a  function $f:\NN\to\NN$  \emph{computable} (or {\it recursive}), if there exists an algorithm $\cA$ which, upon input $n$, outputs $f(n)$. Computable functions of several integer variables are defined in the same way. 

A function $f:W\to\NN$, which is defined on a subset $W\subset \NN$, is called {\it partial recursive} if  there exists an
algorithm $\cA$ which outputs $f(n)$ on input $n\in W$, and runs forever if the input $n\notin W$.

\subsection{Time complexity of a problem.} For an algorithm $\cA$ with input $w$ the {\it running time}  is the number
of steps $\cA$ makes before terminating with an output. 
The {\it size} of an input $w$ is the number of dyadic bits required to specify $w$. Thus for $w\in\NN$, the size of $w$ is
the integer part of $\log_2 w$. 
The {\it running time of $\cA$} is
the function 
$$T_\cA :\NN \to\NN$$
 such that
$$T_\cA(n) = \max\{\text{the running time of }\cA(w)\text{ for inputs } w \text{ of size }n\}.$$
In other words, $T_\cA(n)$ is the worst case running
time for inputs of size $n$.
For a computable function $f : \NN\to \{0,1\}$  the time complexity of $f$ is said to have an
upper bound $T(n)$ if there exists an algorithm $\cA$ with running time bounded
by $T(n)$ that computes $f$.

\subsection{Computable and semi-computable sets of naturals numbers}
A set $E\subseteq \NN$ is said to be \defin{computable} if its characteristic function 
$\chi_E:\NN\to\{0,1\}$ is computable. That is, if there is an algorithm $\cA:\NN \to \{0,1\}$ that, upon input $n$,  halts and outputs $1$ if $n\in E$ or $0$ if $n\notin E$. Such an algorithm allows to \emph{decide} whether or not a number $n$ is an element of $E$. Computable sets are also called \defin{recursive} or \defin{decidable}. 

Since there are only countably many algorithms, there exist only countably many computable subsets of $\NN$. A well known ``explicit'' example of a non computable set is given by the \emph{Halting set} 
$$H:=\{i \text{ such that }\cA_{i}\text{ halts}\}.$$
 Turing \cite{Tur} has shown that  there is no algorithmic procedure to decide, for any $i\in\NN$, whether or not the algorithm with G\"odel number $i$, $\cA_{i}$, will eventually halt. 
 
  On the other hand, it is easy to describe an algorithmic procedure which, on input $i$, will halt if $i\in H$, and will run forever if $i\notin H$. Such a procedure can informally be described as follows: {\sl on input $i$ emulate the algorithm $\cA_i$; if $\cA_i$ halts then  halt.}
  
  In general, we will say that a set $E\subset \NN$ is 
  \defin{lower-computable} (or 
  \defin{semi-decidable}, or \defin{recursively enumerable}) if there exists an algorithm $\cA_E$
  which on an input $n$  halts if $n\in E$, and never halts otherwise. Thus, the algorithm $\cA_E$ 
  can verify the inclusion $n\in E$, but not the inclusion $n\in E^c$.
We say that $\cA_E$ {\it semi-decides} $n\in E$ (or semi-decides $E$).
 The complement of a lower-computable set is called \defin{upper-computable}. 
 
 The following is an easy excercise:
 \begin{prop}
 A set is computable if and only if it is simultaneously upper- and lower-computable.
 \end{prop}
 
\subsection{Computability over the reals}  Strictly speaking, algorithms only work on natural numbers, but this can be easily extended to the objects  of any countable set once a bijection with integers has been established.  The operative power of an algorithm on the objects of such a numbered set obviously depends on what can be algorithmically recovered from their numbers.  For example,  the set $\mathbb{Q}$ of rational numbers can be injectively
numbered $\mathbb{Q}=\{q_0,q_1,\ldots\}$ in an \emph{effective} way: the
number $i$ of a rational $a/b$ can be computed from $a$ and $b$, and vice
versa. The abilities of algorithms on integers are then transferred to the rationals. For instance, algorithms can perform algebraic operations and decide whether or not $q_{i}>q_{j}$ (in the sense that the set $\{(i,j):q_{i}>q_{j}\}$ is decidable). 

Extending algorithmic notions to functions of real numbers was pioneered by Banach and Mazur \cite{BM,Maz}, and
is now known under the name of {\it Computable Analysis}. Let us begin by giving the definition of a computable real
number, going back to the seminal paper of Turing \cite{Tur}.

\begin{definition}A real number $x$ is called

\begin{itemize}
\item \defin{computable} if there is a computable function $f:\NN \to \QQ$ such that $$|f(n)-x|<2^{-n};$$ 
\item \defin{lower-computable} if there is a computable function $f:\NN \to \QQ$ such that 
$$f(n)\nearrow x;$$ 
\item \defin{upper-computable} if there is a computable function $f:\NN \to \QQ$ such that 
$$f(n)\searrow x.$$ 
\end{itemize}
\end{definition}

Algebraic numbers or  the familiar constants such as $\pi$, $e$, or the Feigembaum constant are all computable. However, the set of all computable numbers $\RR_C$ is necessarily
countable, as there are only countably many computable functions.

We also remark that if $x$ is lower-computable then there is an algorithm to semi-decide the set $\{q_{i}<x\}$: just compute $f(n)$ for each $n$  and halt if $q_{i}<f(n)$.  In other words, the set $\{q \in \QQ:q<x\}$ is lower-computable. The converse is also obviously true: 

\begin{prop}
If $E\subset \QQ$ is lower-computable and $x=\sup E<\infty$, 
then  $x$ is lower-computable.

\end{prop}

In the same way as there exist lower-computable sets which are not computable, there exists lower-computable numbers which are not computable. The usual construction is as follows: let $(a_{i})_{i}$ be an algorithmic enumeration (without repetitions) of a lower-computable set $A$ which is not computable. For instance, we can take
$$A=\{i\in\NN\text{ such that } \cA_i\text{ halts}\}.$$
Define 

$$
q_{n}=\sum_{i=0}^n 2^{-a_{i}-1}.
$$

Clearly, $(q_{n})_{n}$ is a computable non-decreasing sequence of rational numbers. Being bounded by 1, it converges. The limit, say $x$, is then a lower-computable number. It $x$ were computable, it would be possible to compute the binary expansion of $x$ which, in turn, would allow to decide the set $A$. 

We also note:
\begin{prop}
A real number $x$ is computable if and only if it is simultaneously lower- and upper-computable.
\end{prop}

\begin{proof}
Let us assume that $x$ is both lower- and upper-computable. Thus there exist algorithms $\cA_1$ and $\cA_2$
which compute sequences of rationals $a_j$ and $b_j$ respectively with
$$a_j\nearrow x\text{ and }b_j\searrow x.$$
Consider the algorithm $\cA$ which on the input $n$ emulates $\cA_1$, $\cA_2$ to find the first $k(n)$ such
that $|a_{k(n)}-b_{k(n)}|<2^{-n}$, and then outputs $a_{k(n)}$. Then $f(n)=a_{k(n)}$ is a computable function such that
$|f(n)-x|<2^{-n}$ and hence $x\in\RR_C$.

The other direction is trivial.
\end{proof}

\subsection{Uniform computability} In this paper we will use algorithms to define \emph{computability} notions on more general objects. Depending on the context, these objects will take particulars names (computable, lower-computable, etc...) but the definition will always follow the scheme:

\medskip
\noindent
\emph{an object $x$ is \emph{computable} if there exists an algorithm $\cA$
satisfying the \emph{property} P($\cA,x$)}.

\medskip\noindent
For example, a real number $x$ is \emph{computable} if there
exists an algorithm $\cA$ which computes a function $f:\NN\to \QQ$ satisfying $|f(n)-x|<2^{-n}$ for all $n$.
 Each time such definition is made, a \emph{uniform version} will be implicitly defined:
 
\medskip
\noindent
\emph{the objects $\{x_\gamma\}_{\gamma\in\Gamma}$ are computable \defin{uniformly on
a countable set $\Gamma$} if there exists an algorithm $\cA$ with an input $\gamma\in\Gamma$, such that for all
$\gamma\in \Gamma$, $\cA_\gamma:=\cA(\gamma,\cdot)$ satisfies the \emph{property} P($\cA_\gamma,x_\gamma$)}.

\medskip
\noindent
 In our example, a sequence of reals $(x_i)_i$ is \emph{computable uniformly in $i$} if there exists
$\cA$ with two natural inputs $i$ and $n$ which computes a function $f(i,n):\NN\times\NN\to\QQ$ such that for all $i\in \NN$, the values of the function $f_{i}(\cdot):=f(i,\cdot)$ satisfy

$$|f_{i}(n)-x_{i}|<2^{-n}\text{ for all } n \in\NN.$$

\subsection{Computable metric spaces}The above definitions equip the real numbers with a computability structure. This can be extended to virtually any separable metric space, making them \emph{computable metric spaces}. We now give a short introduction. For more details, see \cite{Wei}.

\begin{definition}
A \defin{computable metric space} is a triple $(X,d,\S)$ where:
\begin{enumerate}
\item $(X,d)$ is a separable metric space,
\item $\S=\{s_i:i\in\N\}$ is a dense sequence of points in $X$,
\item $d(s_i,s_j)$ are computable real numbers, uniformly in $(i,j)$.
\end{enumerate}
\end{definition}

The  points in $\S$ are called  \defin{ideal}. 
\begin{example}
A basic example is to take the space $X=\RR^n$ with the usual notion of Euclidean distance $d(\cdot,\cdot)$, and 
to let the set $\cS$ consist of points $\bar x=(x_1,\ldots,x_n)$ with rational coordinates. In what follows,
we will implicitly make these choices of $\cS$ and $d(\cdot,\cdot)$ when discussing computability in $\RR^n$.
\end{example}

\begin{definition}A point $x$ is \defin{computable} if there is a computable function
$f:\NN\to\NN$  such that $$|s_{f(n)}-x|<2^{-n}\text{ for all }n.$$
\end{definition}

If $x\in X$ and $r>0$, the metric ball $B(x,r)$ is defined as 
$$B(x,r)=\{y\in X:d(x,y)<r\}.$$
Since the set $\cB:=\{B(s,q):s\in\S,q\in\Q, q>0\}$ of \defin{ideal balls} is countable, we can fix an enumeration
$\B=\{B_i:i\in\N\}$. 

\begin{proposition}\label{p.comp.semidec}
A point $x$ is computable if and only if the relation $x\in B_{i}$ is semi-decidable, uniformly in $i$.
\end{proposition}
\begin{proof}
Assume first that $x$ is computable. We have to show that there is an algorithm $\cA$ which inputs a natural number $i$ and halts
 if and only if $x\in B_{i}$. Since $x$ is computable, for any $n$ we can produce an ideal point $s_{n}$ satisfying $|s_{n}-x|<2^{-n}$. 
The algorithm $\cA$ work as follows: upon input $i$, it computes the center and radius of $B_{i}$, say $s$ and $r$. It then
searches for $n\in\NN$ such that
$$d(s_{n},s)+2^{-n}<r.$$ 
Evidently, the above inequality will hold for some $n$ if and only if $x\in B_i$.

Conversely, assume that the relation  $x\in B_{i}$, s semi-decidable uniformly in $i$. 
To produce an ideal point $s_{n}$ satisfying $|s_{n}-x|<2^{-n}$, we only need to enumerate all ideal balls of radius $2^{-{n+1}}$ until one containing 
$x$ is found. We can take $s_{n}$ to be the center of this ball. 
\end{proof}

\begin{definition}
 An open set $U$ is called \defin{lower-computable} if there is a computable function $f:\NN\to\NN$
 such that $$U=\bigcup_{n\in \NN}B_{f(n)}.$$ 
\end{definition}

\begin{example} Let $\eps>0$ be a lower-computable real. Then the ball $B(0,\eps)$ is a lower-computable open set. 
Indeed: $B(s_0,\eps)=\bigcup_{n}B(0,q_{n})$, where $(q_{n})_{n}$ is the computable sequence converging to $\eps$ from below. 
\end{example}

It is not difficult to see that finite intersections or infinite unions of (uniformly) lower-computable open sets are again lower computable. 
As in Proposition (\ref{p.comp.semidec}), one can show that the relation $x\in U$ is semi-decidable for a computable point $x$ and 
an open lower-computable set.

We will now introduce computable functions. Let $X'$ be another computable metric space with idea balls $\cB'=\{B_i'\}.$

\begin{definition}A function $f:X\to X'$ is \defin{computable} if the sets $f^{-1}(B'_i)$ are lower-computable open, uniformly in $i$. 
\end{definition}

An immediate corollary of the definition is:

\begin{prop}
Every computable function is continuous.
\end{prop} 

The above definition of a computable function is concise, yet not very transparent. To give its $\eps-\delta$ version, we need another concept.
We say that a function $\phi:\NN\to\NN$ is an {\it oracle} for $x\in X$ if
$$d(s_{\phi(m)},x)<2^{-m}.$$
An algorithm may {\it query} an oracle by reading the values of the function $\phi$ for an arbitrary $n\in\NN$.
We have the following:

\begin{prop}
A function $f:X\to X'$ is computable if and only if
there exists an algorithm $\cA$ with an oracle for $x\in X$ and an input $n\in\NN$ which outputs $s_n'\in \cS'$
such that $d(s_n',f(x))<2^{-n}.$
\end{prop}
In other words, given an arbitrarily good approximation of the input of $f$ it is possible to constructively 
approximate the value of $f$ with any desired precision.


\subsubsection{Computability of closed sets}
Having dfined lower-computable open sets, we naturally proceed to the following definition.
\begin{definition}\label{d.upper.comp}
A closed set $K$  is  \defin{upper-computable}  if its complement is lower-computable. \end{definition}

Let us look at two examples. Firstly,

\begin{example}
 A closed ideal ball $\cl(B(s,q))$ is clearly upper-computable. To see this, observe that  a  
point $s_{n}$ belongs to $X\setminus \cl(B(s,q))$ if and only if $d(s_{n},s)>q$. Since this last relation is semi-decidable, 
we can enumerate such ideal points. Moreover, for each of them we can also find $q_{n}$ satisfying 
$0<q_{n}<d(s_{n},s)-q$, so that $B(s_{n},q_{n})\subset X\setminus \cl(B(s,q))$. 

\end{example}

Our second example is more interesting:

\begin{example}
\label{example-fille-julia-set}
 Let  $P:\CC\to\CC$ be a computable polynomial of degree $\deg P\geq 2$. Then the filled Julia set $K(P)$ is upper-computable. 
\end{example}

\begin{proof}
Indeed, let $M\in\QQ$ be such that $K(P)\subset B(0,M)$.
 Enumerate the points in $\RR^2$ with rational coordinates $\cS=\{s_n=(a_n,b_n)\}$, and set $\zeta_n=a_n+ib_n$.
For every point $\zeta_n\in \CC\setminus K(P)$ we can identify an iterate $P^l(\zeta_n)\notin B(0,M)$. Moreover,
for such a point we can find $\eps_n\in \QQ$ such that $$P^l(B(\zeta_n,\eps_n))\cap B(0,M)=\emptyset.$$
We can thus algorithmically enumerate a sequence of open ideal balls which exhausts $\RR^2\setminus K(P)$.
\end{proof}

\begin{definition}
A closed set $K$ is \defin{lower-computable} if the relation $K\cap B_{i}\neq \emptyset$ is semi-decidable, uniformly in $i$. 
\end{definition}

\noindent
In other words, a closed set $K$ is lower-computable if there exists an algorithm $\cA$ which enumerates all ideal balls which have  non-empty intersection with $K$.

To see that this definition is a natural extension of lower computability of open sets, we note:
\begin{example}$\-$
\begin{enumerate}

\item  The closure of an ideal ball $\cl(B(s,q))$ is lower-computable. Indeed, $B(s_{i},q_{i})\cap \cl(B(s,q))\neq \emptyset $ if and only if $d(s,s_{n})<q+q_{n}$. 

\item More generally, the closure $\cl(U)$ of any open lower-computable set $U$ is lower-computable since $B_{i}\cap \cl(U)\neq \emptyset$ if and only if there exists $s\in  B_{i}\cap U$. 

\end{enumerate}
\end{example}

The following is  a useful characterization of lower-computable sets:
\begin{proposition}\label{p.comp-closed}
A closed set $K$ is lower-computable if and only if  there exists a sequence of uniformly computable  points $x_{i}\in K$ which is dense in $K$.
 \end{proposition}
\begin{proof}
Observe that, given some ideal ball $B=B(s,q)$ intersecting $K$, the relations $\cl(B_i)\subset B$, $ q_{i}\leq 2^{-k}$ and $B_{i}\cap K\neq \emptyset$ are all semi-decidable and then we can find an exponentially decreasing sequence of ideal balls $(B_{k})$ intersecting $K$. Hence $\{x\}=\cap_k B_k$ is a computable point lying in $B\cap K$. 

The other direction is obvious.
\end{proof}

\begin{example}
\label{Julia-sets-example}
 Let $R$ be a computable rational map of degree $\deg R\geq 2$. Then the Julia set $J(R)$ is lower-computable.
\end{example}
\begin{proof}[Sketch of proof.]
We will use \corref{preimages dense}.
Periodic points of $R$ are computable (by any standard root-finding algorithm) and so are their multipliers.
We can semi-decide whether a periodic point is repelling (if the multiplier is greater than $1$ we will be able to establish this with a certainty
by computing the point and its multiplier precisely enough). 
Therefore, the repelling periodic points of $R$ are computable.
Let $w$ be any such point. The points in $\cup_{k\geq 0} f^{-k}(w)$ are uniformly computable, and dense in $J(R)$.
By Proposition \ref{p.comp-closed}, $J(R)$  is a closed lower-computable set.
\end{proof}

\begin{definition}A closed set is \defin{computable} if it is lower and upper computable.
\end{definition}

Putting together Examples \ref{example-fille-julia-set} and \ref{Julia-sets-example}, we obtain the following theorem of \cite{BBY1}:
\begin{example}
Let $P$ be a computable polynomial with $\deg P\geq 2$, and suppose that $K(P)$ has empty interior, that is, $K(P)=J(P)$. Then
$K(P)$ is a computable set.
\end{example}

\begin{figure}[ht]
\centerline{\includegraphics[width=\textwidth]{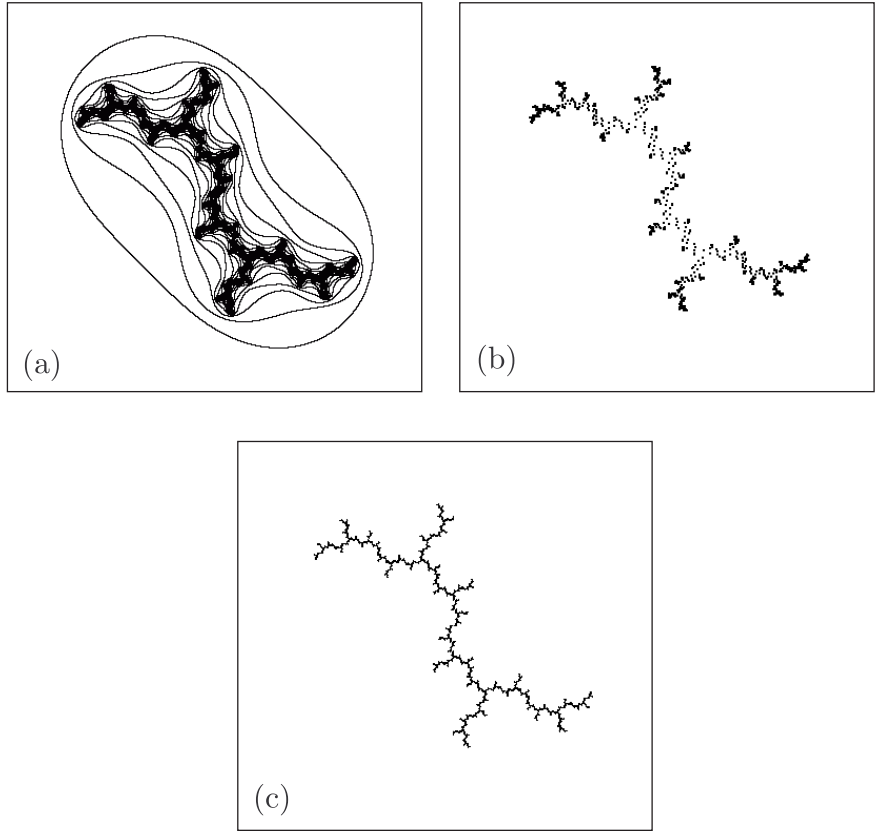}}
\caption{Computing the Julia set of $f(z)=z^2+i$ ($J(f)=K(f)$).
(a) approximating the filled Julia set from above:  the first $15$ preimages of a large disk $D=B(0,R)\supset K(f)$;
(b) approximating the Julia set from below: $\cup_{0\leq k\leq 12}f^{-k}(\beta)$ where $\beta$ is a repelling fixed point in $J(f)$;
(c) a good-quality picture of $J(f)$.
\label{julia-set-figure}
}
\end{figure}

Here is an alternative way to define a computable set. Recall that {\it Hausdorff distance} between two
compact sets $K_1$, $K_2$ is
$$\dist_H(K_1,K_2)=\inf_\eps\{K_1\subset U_\eps(K_2)\text{ and }K_2\subset U_\eps(K_1)\},$$
where $U_\eps(K)=\bigcup_{z\in K}B(z,\eps)$ stands for an $\eps$-neighborhood of a set.
The set of all compact subsets of $M$ equipped with Hausdorff distance is a metric space which we will denote
by $\text{Comp}(M)$. If $M$ is a computable metric space, then $\text{Comp}(M)$ inherits this property; the ideal
points in $\text{Comp}(M)$ are finite unions of closed ideal balls in $M$. We then have the following:

\begin{prop}
A set $K\Subset M$ is computable if and only if it is a computable point in $\text{Comp}(M)$.

Equivalenly, $K$ is computable if there exists an algorithm $\cA$ with a single natural input $n$, which
outputs a finite collection of closed ideal balls $\cl B_1,\ldots, \cl B_{i_n}$ such that
$$\dist_H(\bigcup \cl B_{i_n}, K)<2^{-n}.$$
\end{prop}

\subsection{Computable probability measures}

Let $\M(X)$ denote the set of Borel probability measures over a metric space $X$. We recall the notion of
weak convergence of measures:

\begin{definition}
A sequence of measures $\mu _{n}\in M(X)$ is said to be \defin{weakly convergent} to $\mu\in M(X) $ if $\int f d\mu_{n}\rightarrow \int f d\mu$ for each $f\in C_{0}(X)$. 
\end{definition}

Any smaller family of functions characterizing the weak convergence is called \defin{sufficient}.  It is well-known, that when  $X$ is a compact separable and complete metric space, then so is $\M(X)$.

\noindent
Weak convergence on $M(X)$ is compatible with the notion of \defin{Wasserstein-Kantorovich distance}, defined by:

\begin{equation*}
W_{1}(\mu,\nu)=\underset{f\in 1\text{-Lip}(X)}{\sup}\left|\int f d\mu-\int f d\nu\right|
\end{equation*}
\noindent where $1\mbox{-Lip}(X)$ is the space of Lipschitz functions on $X$, having Lipschitz constant less than one.

  The following result (see \cite{HoyRoj07}) says that, when $X$ is a computable metric space, $\M(X)$ inherits its computability structure.

\begin{proposition}
Let $\cD$ be the set of finite convex rational combinations of Dirac measures supported on ideal points of $X$. Then the triple $(\M(X),W_{1}, \cD)$ is a computable metric space. 
\end{proposition}

\begin{defn}
A \defin{computable measure} is a computable point in $(\M(X),W_{1}, \cD)$. That is, it is a measure which can be algorithmically approximated in the weak sense by \emph{discrete} measures with any given precision.  
\end{defn}

\noindent
As examples of computable measures, consider the Lebesgue measure in $\RR^n$, or any smooth measure in $\RR^n$
with a computable density function.

\ignore{
In general, even if $\mu$ is a computable measure, we can not expect the measure of sets to be computable.  

\begin{example}-
\begin{itemize}
\item[(a)] It is obvious that recursively compact (or open) sets do not need to have computable measure. Just take  Lebesgue measure on $[0,1]$ and let $m$ be an upper-computable number which is not computable. Then the set $[0,m]$ is recursively compact, and its measure is not computable. 
\item [(b)] A computable middle Cantor set with measure $m$ can be obtained as follows: let $1=m_{0}> m_{1} > m_{2} ...$ be a sequence of rational numbers converging to $m$ from above. Set $E_{0}=[0,1]$.  construct $E_{i}$ by removing an open interval of length $2^{-i+1}(m_{i-1}-m_{i})$ from the middle of each component of $E_{i-1}$. The set  $E=\cap E_{i}$ is computable and has measure $m$. For, each removed interval is r.e open so that the complement of $E$ is recursively open, and the endpoints of those intervals form a sequence of computable points which is dense in $E$.
\end{itemize}
\end{example}

}

The following proposition (see \cite{HoyRoj07}) gives a useful characterization of the computability of the measure. 

\begin{proposition}\label{t.comp-measure}
Let $\mu$ be a Borel probability measure. The following statements are equivalent:
\begin{enumerate}
\item $\mu$ is computable,
\item $\mu(B_{i_1}\cup\ldots\cup B_{i_n})$ are lower-computable, uniformly in $i_1,\ldots,i_n$,
\item for any uniformly computable sequence of functions $(f_{i})_{i}$, the integral $\int f_{i}d\mu$ is computable uniformly in $i$.
\end{enumerate}
\end{proposition}

\noindent
We will also need the following fact  (see \cite{Cristobal}), :
\begin{proposition}\label{t.comp-integrals}If $(f_{i})_{i}$ is a uniformly computable sequence of functions, then
the integral operators $$L_{i}:\mathcal{M}(X)\to \R\text{ defined by }L_{i}(\mu):=\int f_{i}d\mu,$$ are uniformly computable.
\end{proposition}

\noindent
To illustrate the concepts we have introduced, 
we end this section by constructing an example of a computable set $E\subset I=[0,1]$ such that Lebesgue measure restricted to it is not computable.  
Indeed, any non-atomic  probability measure assigning positive measure to intervals in $E$, will not be computable.  

\begin{example} Let $\cA_{i}$ be a G{\"o}del ordering of algorithms.  Set $a_{0}=0$ and $a_{i}=a_{i-1}+2^{-i}$. Define the set $S_{i}$ to be:
\[
		S_{i}=\begin{cases}
	(a_{i},a_{i+1})\setminus \{a_{i}+n2^{-j-i}: n=1,...,2^j-1\},         &   \cA_{i} \text{ halts in }j\text{ steps} \\
			 \qquad \qquad   \emptyset,  & \cA_{i} \text{ does not halt}.  \\
			
	\end{cases}
\]

Our set is defined by  $$E=I\setminus \cup_{i}S_{i}.$$

Clearly, $\cup_{i}S_{i}$ is  lower-computable open and thus $E$ is upper computable.  

Let us prove that $E$ is also lower-computable by producing a dense computable sequence of points in $E$.
To this end, we run an algorithm which
at step $j$ simulates all algorithms $\cA_i$, $i\leq j$ for the first $j$ steps (or until they halt).
For every $i$ such that $\cA_i$ does not halt in fewer than $j$ steps it then outputs the set 
$$\{a_{i}+n2^{-j-i}: n=1,...,2^j-1\}.$$ We denote $E_j$ the union of the sets output by the algorithm at step $j$.

It is clear that
$$E=\cl \left(\cup_jE_j\right).$$
Thus $E$ is lower-computable, and hence, computable.

Suppose $\mu$ is a non atomic probability measure on $E$ assigning positive mass to every interval in $E$. 
Then, for each $i$, $\mu(a_{i},a_{i+1})>0$ if and only if $\cA_{i}$ does not halt. 

Let us assume that $\mu$ is a computable measure on $E$. Then, by
 Theorem \ref{t.comp-measure}, the relation $\mu(I)>0$ is semi-decidable for any rational interval $I$.
 Hence the Halting set is upper-computable. Since it is also lower-computable, the Halting set is
 computable. We have thus arrived at a contradiction with the undecidability of the Halting problem.
 
\end{example}

\section{Computability of Brolin-Lyubich measure}

\subsection{Some preliminaries}
In what follows we will require the following  facts. The first theorem is classical, see e.g. \cite{Mil}.

\begin{onequarter}
\it If $f$ is a univalent function on a disk $D\equiv B(z_{0},r)\subset\CC$, then
$$
\dist(f(z_{0}),\partial (f(D)) \geq \frac{1}{4}|f'(z_{0})| r 
$$
\end{onequarter}
 
\ignore{

\noindent
Let us use Koebe 1/4 Theorem to derive the following statement:
\begin{prop}
\label{inverse-computable}
Suppose $f$ is a conformal computable function on some lower-computable open domain $U$. 
Then the inverse $f^{-1}:f(U)\to U$ is also computable. 
\end{prop}
\begin{proof}
Given an orcale for $w\in f(U)$, and $n\in\NN$ we use an exhaustive search to find a rational point $z\in U$ 
with the property
$$B(f(z),2^{-n}f'(z)/4)\ni w.$$
By Koebe 1/4 Theorem, $f^{-1}(w)\in B(z, 2^{-n})$, and thus we can output $z$ as a $2^{-n}$-approximation to the
value of $f^{-1}(w)$.

\end{proof}

}

Considerations of compactness (see \cite{GalHoyRoj09}) imply that there is an algorithmic procedure to
semi-decide whether a given lower-computable open set of probability measures on $\riem$ contains the whole 
$\cM(\riem)$. It will be convenient for us to use a uniform version of this statement:

\begin{proposition}\label{compacity} Let $(U_{i})_{i}$ be a sequence of uniformly lower-computable open subsets of $\M(\riem)$. 
Then the relation $\M(\riem)\subseteq U_{i}$ is semi-decidable, uniformly in $i$. 
\end{proposition} 

\begin{proof}[Sketch of  proof]
It is enough that for any given finite list of ideal balls $\{B_{k_{i}}\}_{i=1}^{m}$,  we can semi-decide the relation
 $$
\M(\riem)\subseteq \bigcup_{i\leq m}B_{k_{i}}.
$$
If this last relation holds, then the union on the right  must contain the elements of any $2^{-n}$-net of $\M(\riem)$, provided that  $2^{-n}$ is less than (half of) the Lebesgue number of the covering  $\{B_{k_{i}}\}_{i=1}^{m}$. Such a net can be computed from a net of $\riem$ and a net of $[0,1]$.

\end{proof}

\subsection{Proof of Theorem A} 
Consider a rational map $R(z)=P(z)/Q(z)$ of degree $d$. 
The coefficients of $P$ and $Q$ form two $(d+1)$-tuples of complex numbers, and we can thus specify $R$ by 
a $(2d+2)$-tuple of coefficients, or a point in $\CC^{2d+2}$. 
It is clear that

\begin{prop}
$R:\riem\to\riem$ is a computable function if and only if there exists a computable point in $\CC^{2d+2}$ which specifies
$R$.
\end{prop}

\noindent
Let us now formulate a precise version of the Theorem A:

 \begin{theorem} 
For a rational map $R$ denote by $\lambda_{R}$ its Brolin-Lyubich measure. 
Then the functional 
 \begin{align*}
 \mathcal{F}:  \CC^{2d+2}  & \to  \M(\riem)\\
                         R & \mapsto   \lambda_{R} 
 \end{align*}
is computable. 
 \end{theorem}

\begin{remark}
In other words, there exists an algorithm $\A$  with an oracle for $\bar v\in\CC^{2d+2}$ and a single natural input $n$ which
outputs a measure $\mu\in \cD$ which has the following property. If $R$ is the rational map with coefficients $\bar v$ then
$$W_{1}(\mu,\lambda_{R})<2^{-n}.$$
\end{remark}

\noindent
Of course, if $R$ is computable, then the oracle can be replaced with  an algorithm computing the coefficients of $R$.

\smallskip
\noindent

\begin{proof}[Proof of the Theorem A] 

Let $R$ be a rational map of degree $d$ and $\phi$ be an oracle for the coefficients of $R$.
 Given $n\in \NN$, we will show how to compute an ideal ball $B_{n}\subset \M(\riem)$ 
with radius $2^{-n}$ containing $\lambda_{R}$.   

Let $U\subset \M(\riem)$ be the set of probability measures which are not invariant with respect to $R$, 
and let $V\subset \M(\riem)$ be the set of probability measures which are not balanced.  
In the following, we show that, using the oracle $\phi$, both $U$ and $V$ are lower-computable open sets.  

Let us introduce a certain fixed, enumerated sequence of Lipschitz computable functions which we will use as test functions.  
Let $\mathcal{H}_0$ be the set of functions of the form:
\begin{equation}\label{lip functions}
\varphi_{s,r,\epsilon}=|1-|d(x,s)-r|^+/\epsilon|^+
\end{equation}

\noindent
where $s$ is a rational point in $\riem$, $r,\epsilon \in\Q$ and $|a|^+=\max\{a,0\}$.
 These are uniformly computable  Lipschitz functions equal to $1$ in the ball $B(s,r)$, to $0$ outside $B(s,r+\epsilon)$ and with
intermediate values in between. 

Let 
\begin{equation}\label{f}
\mathcal{H}=\{\varphi_1,\varphi_2,\ldots,\}
\end{equation} 
be the smallest set of functions containing $\mathcal{H}_0$ and
the constant 1, and closed under $\max$, $\min$ and finite rational linear combinations. Clearly, 
we have:
\begin{proposition}
$\cH$ is a sufficient family of uniformly computable functions.

Moreover, the functions in $\cH$ are of the form $\varphi_n=c_n + g_n$ where $c_n$ is a constant computable from $n$,
 and $g_n$ has a bounded support, and from $n$ one can compute a bound for its diameter. 
\end{proposition}

\begin{lemma}The set $$U:=\{\mu \in \M(\riem) : \mu \text{ is not invariant with respect to }R \}$$ 
is a lower-computable open set.
\end{lemma}
\begin{proof}
We show that $U$ is lower-computable open by exhibiting an algorithm to 
semi-decide whether a probability measure $\mu$ belongs to $U$.  
By Proposition \ref{t.comp-integrals} the numbers 
$$\int \varphi_{i}\, d\mu\text{ and }\int \varphi_{i}\circ R\, d\mu$$
 are uniformly computable. If $\mu$ is not invariant, 
then there exist $j$ such that 
$$\int g_{i}\, d\mu \neq \int g_{i}\circ R \,d\mu$$ and such a $j$ can be found. 
\end{proof}

\begin{lemma}\label{l.isolating}
 The set $$V:=\{\mu \in \M(\riem) : \mu \text{ is not balanced with respect to }R \}$$
 is a lower-computable open set.  
\end{lemma}
\begin{proof}
To semi-decide whether a measure $\mu$ is not balanced, we start by enumerating all the ideal points $z_{i}$ in   $\hat{\mathbb{C}}$ 
which are not critical for $R$.  For each $z_{i}$, we can compute an ideal ball $B_{i}=B(z_{i},r)$ such that $R|_{B_{i}}$ is injective. 
Denote  $\text{Crit}_{R}$ is the set of critical values of $R$.  
Compute a rational number $q$ such that $$0<q<\dist(R(z_{i}),\text{Crit}_{R}).$$
The function $R$ has a conformal inverse branch
$R_{i}^{-1}$ on $B(R(z_{i}),q)$. Compute any rational number
$r$ such that 
$$0<r<\frac{1}{4R'(z_i)}.$$
By Koebe 1/4 Theorem, 
$$B_i=B(z_i,r)\subset R_{i}^{-1}(B(R(z_{i}),q)),$$
so that $R$ is conformal on $B_i$.

 Now, for each $B_{i}$, let $(\varphi^{i}_{j})_{j}$ be the list of test functions supported on $B_{i}$.  
If $\mu$ is not balanced, then there exists $i$ such that 
$$\mu(R(B_{i}))\neq d\cdot\mu(B_{i})$$
 which means that there exists $l$ such that  
$$\int \varphi^{i}_{l}\circ R_{i}^{-1}d \mu \neq d\cdot \int \varphi^{i}_{l} d\mu.$$
By Proposition \ref{t.comp-integrals}, the numbers  $\int \varphi^{i}_{j}\circ R_{i}^{-1}d \mu$ and $d \int \varphi^{i}_{j} d\mu$ 
are uniformly computable, and thus $l$ can be found. 
\end{proof}

It follows that the open set $\mathcal{U}=U\cup V$ of measures which are either not invariant or not balanced is 
lower-computable with an oracle $\phi$.  Its complement is  the singleton $\{\lambda_{R}\}$. 
To compute $\lambda_{R}$ with precision $2^{-n}$, enumerate all the ideal balls $B_{n}\subset \M(\riem)$ of radius $2^{-n}$ 
and semi-decide  the relation $\{\lambda_{R}\}\subset B_{n}$. This is possible because 
$$\{\lambda_{R}\}\subset B_{n}\iff \M(\riem)\subset \mathcal{U}\cup B_{n},$$ and the last relation is semi-decidable by 
Proposition \ref{compacity}. 

\end{proof}

\subsection{A comparison of rates of convergence.}

In \cite{BY-MMJ} two of the authors have shown:
\begin{thm}\label{BY-thm}
There exists a computable quadratic polynomial  $f_c(z)=z^2+c$ whose Julia set $J_c$ is not computable.
 \end{thm}

Together with Theorem A this statement has the following amusing consequence:
\begin{thm}[\bf Incommensurability of rates of convergence]
For a polynomial $f_c$ and $z\in \CC$ denote
$$\phi_1(n)=\dist_H((f_c)^{-n}(z), J_c)\text{ and }\phi_2(n)=W_1\left(\frac{1}{2^{n}}\sum_{w\in (f_c)^{-n}(z)} \delta_w\; ,\lambda_c\right),$$
where $\lambda_c$ is the Brolin-Lyubich measure of $f_c$. Even though both $\phi_1$ and $\phi_2$ converge to $0$ as $n\to\infty$,
there exists a parameter $c$ such that there {\bf does not exist any computable function} $F:\RR\to\RR$ such that $F(0)=0$ and
$$\phi_1(n)\leq F\circ \phi_2(n).$$

\end{thm}

\begin{proof}
If $c$ is computable, Theorem A implies the computability of $\lambda_c$. Hence, $\phi_2(n)$ is a computable function. On the other hand,
if there exists a computable bound $\phi_1(n)\underset{n\to\infty}{\longrightarrow}0$, then $J_c$ is a computable set. Therefore, such a bound
cannot exist for a parameter $c$ as in \thmref{BY-thm}.

\end{proof}

As an illustration, consider Figure \ref{fig-measure}. The Julia set of a quadratic polynomial is rendered in gray. This particular polynomial
can be written in the form $P_\theta(z)=z^2+\exp(2\pi i\theta)z$ for $\theta=(\sqrt{5}+1)/2$. The preimage $(P_\theta)^{-12}(z)$ 
(highlighted in black) for a point $z\in J(P_\theta)$
gives an excellent approximation of $\lambda$, but a very poor approximation of the whole Julia set.

\begin{figure}[ht]
\centerline{\includegraphics[width=0.7\textwidth]{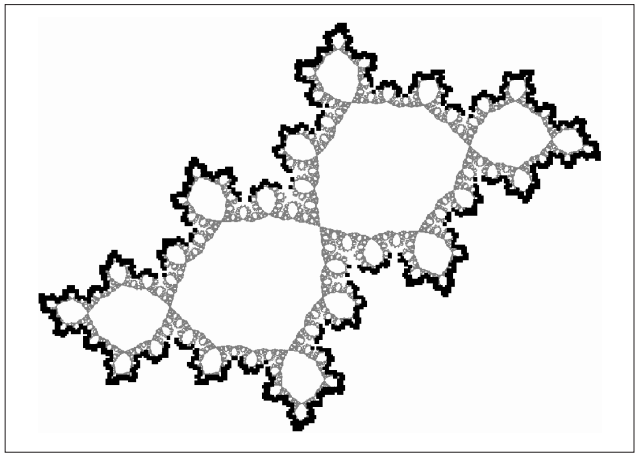}}
\caption{The Julia set of $P_\theta(z)=z^2+\exp(2\pi i\theta)z$ for $\theta=(\sqrt{5}+1)/2$. The set $(P_\theta)^{-12}(z)$ is indicated in black for a point 
$z\in J(P_\theta)$.
\label{fig-measure}
}
\end{figure}

\subsection{Proof of Theorem B} 

For a given point $z\in \CC$, set $$\lambda_{z,m}=\frac{1}{d^m}\sum_{w\in R^{-m}(a)}\delta_{w}.$$  The following result is due to 
Drasin and Okuyama \cite{okuyama}, and, in more generality, to Dinh and Sibony  \cite{DinhSib}.


\begin{theorem}[\cite{okuyama,DinhSib}]
For each $R$ there are constants $\alpha=\alpha(R)$ and $A=A(R)$ such that 
for every point $z\in\riem$, except at most two, and for every $f\in 1-Lip$ the following holds:
$$
\left| \int f\, d\lambda_{z,m} - \int f\, d\lambda \right| < A \alpha^{-m}.
$$
Note that  $A$ and $\alpha$ are independent of $n, z$ and $f$. 
\end{theorem}

We have then that, taking $m=C n$ for some constant $C=C(A,\alpha)$, 
$$
W_{1}(\lambda_{z,m},\lambda)<2^{-n}.
$$
 
So that in order to compute a $2^{-n}$-approximation of $\lambda_{R}$, it is enough to compute approximations of the $d^{m}$ pre-images of $z$ by $R^m$. Since each pre-image can be computed in time polynomial in $d^m$, the entire computation can be achieved in time $O(2^{c n})$ for a $c=c(R)$.

 \subsection{A counter-example}
In view of the above results, it is natural to ask whether a measure of maximal entropy of a computable dynamical system is 
{\it always} computable.
The example below will show that this need not be the case. We will construct a map 
$$T:S^1\times S^1\to S^1\times S^1$$
with the following properties:
\begin{itemize}
\item[(a)] $T$ is a computable function;
\item[(b)] $T$ has a measure of maximal entropy;
\item[(c)] {\it every} measure of maximal entropy of $T$ is non-computable.
\end{itemize}

We first recall a construction \cite{GalHoyRoj09}:
\begin{prop}[\cite{GalHoyRoj09}]
\label{bomb}
There exists a computable transformation $T_1:S^1\to S^1$ for which {\it every} invariant measure is non-computable. 
\end{prop}

To prove this, we need the following  facts: 
\begin{prop}
\label{comp points}
Let $\mu$ be a computable Borel measure on a computable metric space $X$. Then the support of $\mu$ contains a computable point $x\in X$.
\end{prop}

\begin{proof}[Sketch of proof]
We outline the proof here and leave the details to the reader.
First, for each ideal ball $B=B(x,r)$ set  
$$\psi_B\equiv \phi_{x,r/2,r/2}$$
as in (\ref{lip functions}).
An  exhaustive search can be used to find
a sequence of ideal balls $B^i=B(x_i,r_i)$ with the following properties:
\begin{itemize}
\item $r_i\leq 2^{-i}$;
\item $B(x_i,r_i)\subset B(x_{i-1},2r_{i-1})$;
\item $\int\psi_{B^i} d\mu>0$.
\end{itemize}
The  algorithm can then be used to compute $x=\lim x_i\in\text{supp}(\mu)$. 

\end{proof}

\begin{prop}
\label{lco}
There exists a lower-computable open set $V\subset (0,1)$ such that $(0,1)\setminus V\neq \emptyset$ and 
$V$ contains all computable numbers in $(0,1)$. 
\end{prop}
\begin{proof}[Sketch of proof]
Consider an algorithm $\cA$ which at step $m$ emulates the first $m$ algorithms $\cA_i(i)$, $i\leq m$ with respect to the G{\"o}del ordering
for $m$ steps. That is, the $i$-th algorithm in the ordering is given the number $i$ as the input parameter.
From time to time, an emulated algorithm $\cA_i(i)$ may output a rational number $x_i$ in $(0,1)$.
Our algorithm $\cA$ will output an interval 
$$L_i=(x_i-3^{-i}/2,x_i+ 3^{-i}/2)\cap (0,1)$$
for each term in this sequence.
The union $V=\cup L_i$ is a lower-computable set. It is easy to see from the definition of a computable real that $V\supset \RR_C\cap(0,1)$.
If $x\in \RR_C$ then there is a machine $\cA_n(j)$ that on input $j$ outputs a $3^{-j}/4$-approximation of $x$. Thus 
the execution of $\cA_n(n)$ will halt with an output $q$ such that $|x-q|<3^{-n}/4$, and $x$ will be included in $V$.
 On the other hand, the Lebesgue measure of $V$ is bounded by $1/2$, and thus does not cover all of $[0,1]$.
\end{proof}

\begin{proof}[Sketch of proof of \propref{bomb}]

By \propref{lco}, there exists an open lower-computable set $V\subset [0,1]$ such that the complement $K=[0,1]\setminus V$ contains no computable points.  Since $V$ is lower-computablle, there are computable sequences $\{a_i,b_i\}_{i\geq 1}$ such that $0<a_i<b_i<1$ and $V=\bigcup_i (a_i,b_i)$.

\begin{figure}[ht]
\centerline{\includegraphics[width=0.3\textwidth]{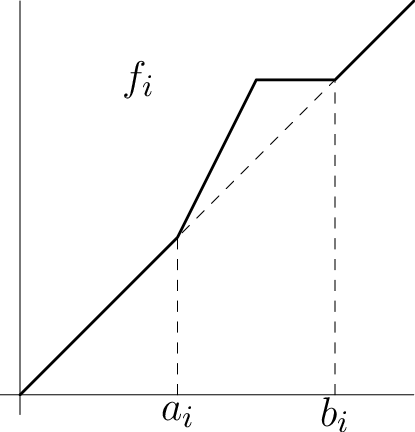}\hspace{0.1\textwidth}\includegraphics[width=0.3\textwidth]{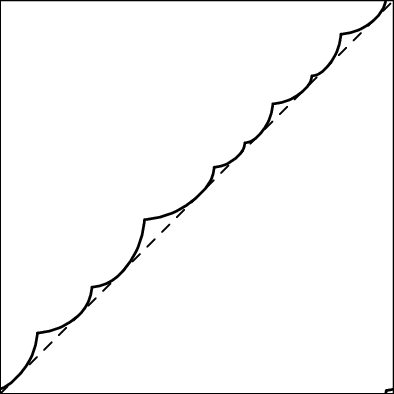}}
\caption{Left: a map $f_i$, right: the map $T_1$\label{mapt1}}
\end{figure}

 Let us define non-decreasing, uniformly computable functions $f_i:[0,1]\to[0,1]$ such that 
$$f_i(x)>x\text{ if }x\in (a_i,b_i)\text{ and }f_i(x)=x\text{ otherwise}.$$
 For instance, we can set
$$f_i(x)=2x-a_i\text{ on }\left[a_i,\frac{a_i+b_i}{2}\right],\text{ and }$$
$$f_i(x)=b_i\text{ on }\left[\frac{a_i+b_i}{2},b_i\right].$$ 
As neither $0$ nor $1$ belongs to $K$, there is a rational number $\epsilon>0$ such that $K\subseteq [\epsilon,1-\epsilon]$. 
Let us define $f:[0,1]\to \R$ by 
$$f(x)=\left\{
\begin{array}{l}
x\text{ on }[\epsilon,1-\epsilon],\\
2x-(1-\epsilon)\text{ on }[1-\epsilon,1]\\
\epsilon\text{ on }[0,\epsilon]
\end{array}
\right.
$$
We then define
$t(x):[0,1]\to\RR$ by
$$t(x)=\frac{f}{2}+\sum_{i\geq 2} 2^{-i}f_i.$$
By construction, the function  $t(x)$ is computable and non-decreasing, and $t(x)>x$ if and only if $x\in [0,1]\setminus K$. 
As $$t(1)=f(1)=1+\epsilon=1+t(0),$$
we can take the quotient 
$$T_1(x)\equiv t(x)\mod \ZZ.$$

It is easy to see that $T_1$ moves all points towards the set $K$. 
More precisely, every point $x\in K$ is fixed under $T_1$, and the orbit of every point $x\notin K$ converges to 
$\text{inf}\{y\in K\cap [x,1]\}$.
Further, for any interval $J\Subset U$, all but finitely many $T_1$-translates of $J$ are disjoint from $J$. Hence, no finite invariant
measure of $T_1$ can be supported on $J$. Thus the support of every $T_1$-invariant measure is contained in $K$.  
 By \propref{comp points}, no such measure can be computable.

\end{proof}

We are now equipped to present the counter-example $T$. We define $T_2:S^1\to S^1$ by $T_2(x)=2x\mod\ZZ$, and set
$$T=T_1\times T_2.$$
Firstly, by the same reasoning as above, every invariant measure $\mu$ of $T$ is supported on $K\times S^1$ and hence is not computable 
by \propref{comp points}. 
On the other hand, $T$ possesses invariant measures of maximal entropy. Indeed, 
let $\nu$ be any invariant measure of $T_1$ and $\lambda$ the Lebegue measure on $S^1$. Setting $\mu=\nu\times\lambda$, we have
$$h_\mu(T)=h_{\text{top}}(T)=\log 2.$$

\section{Harmonic Measure}

\subsection{Proof of Theorem C} 

Let us start with several definitions.

\begin{defn}
We recall that a compact set $K\subset\riem$ which contains at least two points is {\it uniformly perfect} if the moduli of the ring domains 
separating $K$ are bounded from above. Equivalently, there exists some $C>0$ such that for any 
$x\in K$ and $r>0$, we have 
$$\left(B(x,Cr)\setminus B(x,r)\right)\cap K=\emptyset\implies K\subset B(x,r).$$ In particular, every connected set is uniformly perfect. 
\end{defn}

It is  known that:

\begin{thm}[see \cite{MR92}]
The  Julia set of a polynomial $P$ of degree $d\geq 2$ is a uniformly perfect compact set.
\end{thm}

\noindent
Recall that the {\em logarithmic capacity} $\text{Cap}(\cdot)$ has been defined in Section \ref{sec:23}.
We next define:

\begin{defn}
Let $\Omega\subset \riem$ be an open and connected domain and set $J=\partial\Omega$. We say that $\Omega$ satisfies the {\it capacity density condition}
if there exists a constant $C>0$ such that
\begin{equation}\label{cap-condition}
\text{Cap}(B(x,r)\cap \partial \Omega)\geq Cr \qquad \text{ for all } x\in \partial \Omega\text{ and }r\leq r_0.
\end{equation}
\end{defn}

\noindent
We note:

\begin{thm}[see Theorem 1 in \cite{Pom79}]
Condition \eqref{cap-condition} is equivalent to uniform perfectness of $\partial\Omega$.
\end{thm}

\noindent
The celebrated result of Kakutani \cite{kakutani1944two}, gives a connection between Brownian motion 
and the Harmonic measure. 

\begin{thm}\cite{kakutani1944two}
Let $K\subset \riem$ be a  compact set with a connected complement $\Omega$. 
Fix a point $x\in\Omega$ and let $B_t$ denote a Brownian path started at $x$.
Let the random variable $T$ be the first moment when $B_t$ hits $\partial\Omega$, and let
$\omega_x = \omega_{\Omega,x}$ denote the harmonic measure corresponding to $x$. 
Then for any measurable function $f$ on $\partial \Omega$,
$$
\int f d\omega_x = \EE[f(B_T)].
$$
\end{thm}

\noindent
In \cite{BB07} the following computable version of the Dirichlet problem has been proved:

\begin{thm}
\label{dirichlet-compute}
Let $K\subset \riem$ be a  compact computable set with a connected complement $\Omega$. 
Let  $x\in\Omega$ be any point in $\Omega$ and let $B_t$ denote a Brownian path started at $x$.
There is an algorithm $\cA$ that, given access to $K$, $x$, and a precision parameter $\ve$, outputs
an $\ve/4$-approximate sample $\be_\ve$ from a random variable $B_{T_\ve}$ where $T_{\ve}$ is a stopping rule on $B_t$ that always satisfies 
$$
\ve/2 < \dist(B_{T_\ve},K) < \ve.
$$
\end{thm}


\noindent
In other words, we are able to stop the Brownian motion at distance $\approx \ve$ from the boundary. 
We now formulate the following proposition that is a reformulation of Theorem~C:

\begin{proposition}\label{prop:harm}
Let $\Omega\subset \riem$ be the complement of a computable compact set $K$  and $x_{0}$ be a point in $\Omega$. Suppose $\Omega$ is connected and 
satisfies the capacity density condition.
Then the harmonic measure $\omega_{\Omega,x_0}$ is computable with an oracle for $x_0$. 

\end{proposition}

\begin{proof}
Fix $x_0\in\Omega$.  As before, we denote by
$(B_t)$ the Brownian Motion started at $x_0$
 and set
 $$T=\min\{t\ :\ B_t\in\partial\Omega\}.$$ 
We will use Theorem~\ref{dirichlet-compute} together with the capacity density condition to prove
Proposition~\ref{prop:harm}.

The capacity density condition implies the following (see \cite{Garnett-Marshall}, page 343):

\begin{prop}
\label{first hit}
There exists a constant $\nu=\nu(C)$ (with $C$ as in the capacity density condition)
such that for any $\eta>0$ the following holds. Let $y\in \Omega$ be a point such that $\dist(y,\omega)\le \eta$, and 
let $B^\omega$ be a Brownian Motion started at $y$. Let $$T^y:=\min\{t~:~B^y_t\in \partial \Omega\}$$ be the first time 
$B^y$ hits the boundary of $\Omega$. Then 
\beq
\label{eq:cap}
\PP[|B^y_{T_y}-y|\geq 2\eta]<\nu.\eeq
\end{prop}
 In other words, there is at least a constant probability that the first point where $B^y$ hits the boundary
 is close to the starting point $y$. 

Now let $f$ be any function on $K$ satisfying the $1$-Lip condition. Our goal is to compute 
$$\int f\,d\omega=\EE[f(B_T)].$$
within any prescribed precision parameter $\de$. Note that 
$$
\EE[f(B_T)] = \EE_{B_{T_\ve}}[\EE[f(B_T)|B_{T_\ve}|]].
$$
We first claim that we can compute an $\ve$ such that 
\beq \label{cl:1}
| \tilde{f}_{2\ve}(B_{T_\ve}) - \EE[f(B_T)|B_{T_\ve}|]| < \de/2.
\eeq
Here $T_\ve$ is given by the any stopping rule as in Proposition~\ref{dirichlet-compute}, and 
$\tilde{f}_{2\ve}(B_{T_\ve})$ is {\em any} evaluation of $f$ in a $2 \ve$-neighborhood of $B_{T_\ve}$ (note 
that $f$ itself is not defined on $B_{T_\ve}\notin K$). 

Let $M$ be a universal bound on the absolute value of $f$. 
By \eqref{eq:cap} we can compute an $\ve<\de/20$ such that if $y$ is $\ve$-close to $K$, the probability 
that $|B^y_{T_y}-y|>\de/10$ is smaller than $\de/10M$. We split the probabilities into two cases: one where
$B_T$ stays $\de/10$-close to $B_{T_\ve}$, and the complementary case. By \eqref{cl:1} we have
\begin{multline*}
| \tilde{f}_{2\ve}(B_{T_\ve}) - \EE[f(B_T)|B_{T_\ve}|]|  = \\ | \tilde{f}_{2\ve}(B_{T_\ve}) - \EE[f(B_T)|B_{T_\ve}|,|B_T-B_{T_\ve}|<\de/10]|\cdot \PP[|B_T-B_{T_\ve}|<\de/10] + \\ | \tilde{f}_{2\ve}(B_{T_\ve}) - \EE[f(B_T)|B_{T_\ve}|,|B_T-B_{T_\ve}|\ge\de/10]|\cdot \PP[|B_T-B_{T_\ve}|\ge\de/10]\le\\
(\de/6)\cdot 1+M\cdot (\de/10M) < \de/2. 
\end{multline*}
 To complete the proof of the proposition it remains to note that given a $\be_\ve$ that 
 $\ve/4$-approximates $B_{T_\ve}$ as in Theorem~\ref{dirichlet-compute}, we can evaluate $\tilde{f}_{2\ve}(B_{T_\ve})$ by 
 evaluating $\tilde{f}_{3\ve/2}(\be_\ve)$ (by evaluating $f$ at any point in a $3\ve/2$-neighborhood of $\be_\ve$). In this way, we obtain
 \beq \label{cl:2}
\EE_{\be_\ve} | \tilde{f}_{3\ve/2}(\be_\ve) - \EE[f(B_T)]| < \de/2.
\eeq
Thus, being able to evaluate $\tilde{f}_{3\ve/2}(\be_\ve)$ with precision $\de/2$ suffices.
\end{proof}

%
%
%
%
%

\subsection{A counter-example}

As demonstrated by the following example, even for a computable regular domain, the harmonic measure is not necessarily computable. 
Thus the capacity density condition in Theorem~C is cruicial. 

For $a,b\in \mathbb R$, we denote by $[a,b]$ the shortest arc of the unit circle between $e^{2\pi ia}$ and $e^{2\pi ib}$. As before, let $\cA_{i}$ be the G{\"o}del ordering of
algorithms.
Define a collection of subsets of the circle as follows.
If $\cA_{n}$  halts in $l$ steps, we set $j=\max(l,8n)$ and denote by
$$L_n:=L_n^j:=\cup_{k=1-2^j+2^{j-8n}}^{2^j-2^{j-8n}-1}[2^{-n}+k \cdot 2^{-2n-j},2^{-n}+k\cdot 2^{-2n-j}+\exp\left(-2^{2n+2j}\right)].$$
Otherwise, if $\cA_{n}$  does not halt, we denote
$$L_n:=L_n^\infty:=[2^{-n}-2^{-2n}+2^{-10n},2^{-n}+2^{-2n}-2^{-10n}]$$
(see \figref{fig-gates}).

\begin{figure}[ht]
\centerline{\includegraphics[width=\textwidth]{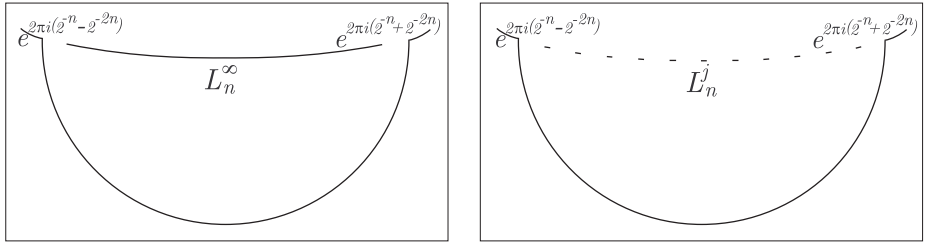}}
\caption{Left: a set $L_n^\infty$ corresponding to an algorithm $\cA_n$ which does not halt; right: a set $L_n^j$ corresponding to $\cA_n$ which halts in $j$ steps.
\label{fig-gates}
}
\end{figure}

Let  $D_n$ denote the disk whose diameter is given by the points $\exp(2\pi i (2^{-n}-2^{-2 n}))$ and  $\exp(2\pi i (2^{-n}+2^{-2 n}))$. Let $$\Lambda:={\mathbb D}\cup\bigcup_{n=10}^{\infty}D_n.$$
The domain $\Omega$ is obtained by removing the arcs from $L_n$ from  $\Lambda$. To be precise, set
$$\Omega:=\Lambda\setminus\bigcup_{n=10}^{\infty}L_n.$$
and
$$K:=\CC\setminus \Omega.$$

We note:
\begin{prop}
The compact set  $K$ is computable.
\end{prop}
\begin{proof}
Note that 
$$\dist_H(L_n^j,L_n^\infty)<12\cdot 2^{-2n-j}.$$
We can thus compute $L_n$ with an arbitrary precision by emulating $\cA_n$ for sufficiently many steps.

To compute the set $K$ with precision $2^{-m}$, it suffices to compute the first $2^{m+4}$ sets $L_n$ with precision $2^{-(m+4)}$.
\end{proof}

Now let us show that:

\begin{prop}
\label{hnc}
The harmonic measure 
$$\omega:=\omega_{\Omega,0}$$ is not computable. 
\end{prop}

For a set $K_0\subset\{z\ :\ 0 < \delta < |z| < r < 1 \}$ set $$\gamma(K_0):=-\log \text{Cap}(K_0).$$

We need an auxilliary lemma:
\begin{lem}[Theorem 5.1.4 in \cite{Ransford}]
\label{subadditivity}
If $K_1,\dots, K_n$ are compact subsets of the unit disk, then 
$$\frac1{\gamma(K_1\cup\cdots\cup K_n)}\leq\frac1{\gamma(K_1)}+\dots+\frac1{\gamma(K_n)}.$$
\end{lem}

Let $S_n$ be the part of the boundary of the disk $D_n$ lying outside of $\mathbb D$, $S_n:=\partial D_n\setminus \mathbb D$. 
Harmonic measure is always non-atomic, so if $\omega$ is computable, 
then $\omega(S_n)$ is also computable. 
We show:

\begin{proposition}\label{prop:example}
If $\cA_n$ does not halt, then $\omega(S_n)<2^{-9n+2}$. If $\cA_n$ halts, $\omega(S_n)> 2^{-2n-3}$. 
\end{proposition}

\begin{proof}
As before, let $B_t$ be the Brownian motion started at $0$ and let $T$ denote the hitting time of $\partial\Omega$, 
$$T:=\inf\{t\ :\ B_t\in\partial\Omega\}.$$
Let us recall that for $E\subset\partial\Omega$ we have 
$$\omega(E)={\mathbb P}[B_T\in E].$$

Assume now that $\cA_n$ does not halt. Let us introduce  a  new domain $\Omega':={\mathbb C}\setminus[2^{-n}-2^{-2n}, 2^{-n}+2^{-2n}]$ 
and $T'$ be the corresponding hitting time. Observe that if $B_T\in S_n$ then $$B_{T'}\in K_n:=[2^{-n}-2^{-2n}, 2^{-n}-2^{-2n}+2^{-10n}]\cup[2^{-n}+2^{-2n}-2^{-10n}, 2^{-n}+2^{-2n}].$$ Thus 
$$\omega_{\Omega,0}(S_n)={\mathbb P}[B_T\in S_n]\leq{\mathbb P}[B_{T'}\in K_n]=\omega_{\Omega',0}(K_n).$$
The desired estimate is now obtained by mapping $(\Omega',0)$ conformally to $({\mathbb D},0)$. 

Assume that $\cA_n$ halts in $j$ steps. To bound $\omega(S_n)$ in this case, we will use the following estimate on harmonic measure (\cite{Garnett-Marshall}, Equation (III.9.2)):

Let $K_0\subset\{z\ :\ 0 < \delta < |z| < r < 1 \}$. Then
\begin{equation}\label{eq:harm_est}\omega_{{\mathbb D} \setminus K_0,0}(K_0) \leq
\frac{\log \left({1 \over {\delta}}\right)} { \gamma(K_0) +\log (1 - r^2)}\end{equation}

Let $T''$ denote the hitting time of  $\partial{\mathbb D}$ by $B_t$, and let 
$$M_n:=\left\{z\in L_n^\infty: \dist(z,\partial\Omega)>2^{-2n-j-4}\right\}$$
be the part of the arc $L_n^\infty$ lying relatively far away from the boundary. 

Conformally mapping $D_n$ to the unit disk centered at $z\in M_n$ and using the estimate \eqref{eq:harm_est} and 
\lemref{subadditivity},
we obtain that for $z\in M_n$ we have
\begin{equation}\label{eq:harm_est_1}
\omega_{D_n\setminus L_n,z}(L_n)\leq 2^{-j+3}< 1/8
\end{equation}

We will also need $T_1\geq T''$ -- the first hitting time of $\partial D_n$ \emph{after} hitting $\partial D$, and 
$T_2 \le T_1$ -- the first hitting time of $\partial D_n \cup L_n$ \emph{after} hitting $\partial D$. 

Note now that 
\begin{equation}\label{eq:length}
{\mathbb P}[B_{T''}\in M_n]=\text{length}(M_n)/2\pi\geq 2^{-2n-1}
\end{equation}

Let us note that by symmetry and estimate \eqref{eq:harm_est_1}, we have
\begin{multline}\label{eq:one}
{\mathbb P}[B_T\in S_n\ |\ B_{T''}\in M_n] \ge {\mathbb P}[B_{T_1}\in S_n\ |\ B_{T''}\in M_n]-
{\mathbb P}[B_{T_2}\in L_n\ |\ B_{T''}\in M_n] =\\ \frac{1}{2} - {\mathbb P}[B_{T_2}\in L_n\ |\ B_{T''}\in M_n]
\geq
\frac{1}{2}-\max_{z\in M_n}\omega_{D_n\setminus L_n,z}(L_n)>1/4.
\end{multline}

The desired lower estimate on $\omega$ is now obtained by combining \eqref{eq:length} and \eqref{eq:one}.
\end{proof}

We now conclude the proof of \propref{hnc}:
\begin{proof}[Proof of \propref{hnc}]
Assume the contrary, that is, suppose that $\omega$ is computable. For every $n\in\NN$ let $\{U^n_j(z)\}_{j=1}^\infty$ 
be a sequence of functions given by:
$$U^n_j(z)=\left\{\begin{array}{l}
1\text{ if }\dist(z, S_n)<2^{-j};\\
0\text{ if }\dist(z,S_n)>2\cdot 2^{-j};\\
1-2^j(d-2^{-j})\text{ if }d=\dist(z,S_n)\in[2^{-j},2\cdot 2^{-j}]
\end{array}\right.$$
We have:
\begin{itemize}
\item[(a)] the functions $U^n_j(z)$ are computable uniformly in $n$ and $j$.
\end{itemize}
Since $\omega$ is non-atomic,
\begin{itemize}
\item[(b)] for a fixed $n$ we have
$$\int U^n_j(z)d\omega>\omega(S_n)\text{ and }\int U^n_j(z)d\omega\underset{j\to\infty}{\longrightarrow}\omega(S_n).$$
\end{itemize}
Similarly, we can costruct a sequence of functions $L^n_j(z)$ such that 
\begin{itemize}
\item[(c)] the functions $L^n_j(z)$ are computable uniformly in $n$ and $j$;
\item[(d)] for a fixed $n$ we have
$$\int L^n_j(z)d\omega<\omega(S_n)\text{ and }\int L^n_j(z)d\omega\underset{j\to\infty}{\longrightarrow}\omega(S_n).$$
\end{itemize}
We leave the details of the second construction to the reader.

By part (3) of \propref{t.comp-measure},  properties (a) and (c) imply that 
the integrals 
$$\int U^n_j(z)d \omega\text{ and }\int L_j^n(z)d\omega$$
are uniformly computable. Consider and algorithm 
$\cA_\text{halt}$ which upon inputting a natural number $n$ does the following:

\begin{enumerate}
\item $j:=1$;
\item evaluate $u_j,l_j$ such that 
$$|u_j-\int U^n_j(z)d \omega|<2^{-20n}\text{ and }|l_j-\int L_j^n(z)d\omega|<2^{-20n};$$
\item if $u_j<2^{-9n+2}+2^{-19n}$ then output $0$ and halt;
\item if $l_j>2^{-2n-3}-2^{-19n}$ then output $1$ and halt;
\item $j:=j+1$ and go to (2). 
\end{enumerate}

By \propref{prop:example} and properties (b) and (d), we have the following:
\begin{itemize}
\item if $\cA_n$ halts then $\cA_\text{halt}$ outputs $1$ and halts, and  
\item if $\cA_n$ does not halt then $\cA_\text{halt}$ outputs $0$ and halts.
\end{itemize}
Thus $\cA_\text{halt}$ is an algorithm solving the Halting Problem, which contradicts the algorithmic unsolvability of the Halting Problem.
\end{proof}

\bibliographystyle{plain}
\bibliography{biblio}

\begin{thebibliography}{10}

\bibitem{BM}
S.~Banach and S.~Mazur.
\newblock Sur les fonctions caluclables.
\newblock {\em Ann. Polon. Math.}, 16, 1937.

\bibitem{BB07}
I.~Binder and M.~Braverman.
\newblock Derandomization of euclidean random walks.
\newblock In {\em APPROX-RANDOM}, pages 353--365, 2007.

\bibitem{BBY2}
I.~Binder, M.~Braverman, and M.~Yampolsky.
\newblock On computational complexity of {S}iegel {J}ulia sets.
\newblock {\em Commun. Math. Phys.}, 264(2):317--334, 2006.

\bibitem{BBY1}
I.~Binder, M.~Braverman, and M.~Yampolsky.
\newblock Filled {J}ulia sets with empty interior are computable.
\newblock {\em Journ. of FoCM}, 7:405--416, 2007.

\bibitem{BY}
M.~Braverman and M.~Yampolsky.
\newblock Non-computable {J}ulia sets.
\newblock {\em Journ. Amer. Math. Soc.}, 19(3):551--578, 2006.

\bibitem{BY-MMJ}
M~Braverman and M.~Yampolsky.
\newblock Computability of {J}ulia sets.
\newblock {\em Moscow {M}ath. {J}ourn.}, 8:185--231, 2008.

\bibitem{BY-book}
M~Braverman and M.~Yampolsky.
\newblock {\em Computability of {J}ulia sets}, volume~23 of {\em Algorithms and
  {C}omputation in {M}athematics}.
\newblock Springer, 2008.

\bibitem{Brolin}
H.~Brolin.
\newblock Invariant sets under iteration of rational functions.
\newblock {\em Ark. Mat.}, 6:103--144, 1965.

\bibitem{DinhSib}
T.~Dinh and N.~Sibony.
\newblock Equidistribution speed for endomorphisms of projective spaces.
\newblock {\em Math. Ann.}, 347:613--626, 2009.

\bibitem{okuyama}
D.~Drasin and Y.~Okuyama.
\newblock Equidistribution and {N}evanlinna theory.
\newblock {\em Bull. Lond. Math. Soc.}, 39:603–--613, 2007.

\bibitem{GalHoyRoj09}
S.~Galatolo, M.~Hoyrup, and C.~Rojas.
\newblock Dynamics and abstract computability: computing invariant measures.
\newblock {\em Discr. Cont. Dyn. Sys. Ser A}, 2010.

\bibitem{Garnett-Marshall}
J.B. Garnett and D.E. Marshall.
\newblock {\em Harmonic measure}.
\newblock Cambridge University Press, 2005.

\bibitem{HoyRoj07}
M.~Hoyrup and C.~Rojas.
\newblock Computability of probability measures and {M}artin-{L}of randomness
  over metric spaces.
\newblock {\em Information and Computation}, 207(7):830--847, 2009.

\bibitem{kakutani1944two}
S.~Kakutani.
\newblock {Two-dimensional Brownian motion and harmonic functions}.
\newblock In {\em Proc. Imp. Acad. Tokyo}, volume~20, 1944.

\bibitem{Lorenz}
E.~N. Lorenz.
\newblock Deterministic nonperiodic flow.
\newblock {\em J. Atmos. Sci.}, 20:130--141, 1963.

\bibitem{Lyubich}
M.~Lyubich.
\newblock The measure of maximal entropy of a rational endomorphism of a
  {R}iemann sphere.
\newblock {\em Funktsional. Anal. i Prilozhen.}, 16:78--79, 1982.

\bibitem{MR92}
R.~Man{\~e} and L.F. da~{R}ocha.
\newblock Julia sets are uniformly perfect.
\newblock {\em Proc. Amer. Math. Soc.}, 116:251--257, 1992.

\bibitem{Maz}
S.~Mazur.
\newblock {\em Computable {A}nalysis}, volume~33.
\newblock Rosprawy Matematyczne, Warsaw, 1963.

\bibitem{Milnor-attractor}
J.~Milnor.
\newblock On the concept of attractor.
\newblock {\em Commun. Math. Phys}, 99:177--195, 1985.

\bibitem{Mil}
J.~Milnor.
\newblock {\em Dynamics in one complex variable. {I}ntroductory lectures}.
\newblock Princeton University Press, 3rd edition, 2006.

\bibitem{Palis}
J.~Palis.
\newblock A global view of dynamics and a conjecture on the denseness of
  finitude of attractors.
\newblock {\em Ast{\'e}risque}, 261:339 -- 351, 2000.

\bibitem{Pom79}
Ch. Pommerenke.
\newblock Uniformly perfect sets and the {P}oincar{\'e} metric.
\newblock {\em Arch. Math.}, 32:192--199, 1979.

\bibitem{Ransford}
Thomas Ransford.
\newblock {\em Potential theory in the complex plane}, volume~28 of {\em London
  Mathematical Society Student Texts}.
\newblock Cambridge University Press, Cambridge, 1995.

\bibitem{Cristobal}
C.~Rojas.
\newblock {\em Randomness and ergodic theory: an algorithmic point of view}.
\newblock PhD thesis, Ecole Polytechnique, 2008.

\bibitem{Smale}
S.~Smale.
\newblock Differential dynamical systems.
\newblock {\em Bull. Am. Math. Soc.}, 73:747--817, 1967.

\bibitem{Tucker}
W.~Tucker.
\newblock A rigorous {O}{D}{E} solver and {S}male's 14th problem.
\newblock {\em Found. Comp. Math.}, 2:53--117, 2002.

\bibitem{Tur}
A.~M. Turing.
\newblock On computable numbers, with an application to the
  {E}ntscheidungsproblem.
\newblock {\em Proceedings, London Mathematical Society}, pages 230--265, 1936.

\bibitem{Wei}
K.~Weihrauch.
\newblock {\em Computable {A}nalysis}.
\newblock Springer-Verlag, Berlin, 2000.

\end{thebibliography}

\end{document}